\documentclass{article}

\usepackage[T1]{fontenc}
\usepackage[utf8]{inputenc}
\usepackage{lmodern}
\usepackage{microtype}

\usepackage[total={6.2in, 8in}, asymmetric, bindingoffset=0.4in]{geometry}
\usepackage{amsmath}
\usepackage{amssymb}
\usepackage{amsthm}
\usepackage{mathtools}
\usepackage{enumerate}
\usepackage[dvipsnames]{xcolor}

\usepackage{multirow}
\usepackage{empheq}
\usepackage{listings}
\usepackage{color}
\usepackage{bbm}

\usepackage[square,comma,numbers]{natbib}
\usepackage{hyperref}
\usepackage{graphicx} %
\usepackage[font=small,labelfont=bf]{caption}
\usepackage{tikz}
\usetikzlibrary{babel}
\usepackage{subfig}

\numberwithin{equation}{section}
\usepackage{pgfplots}
\pgfplotsset{width=10cm,compat=1.9}

\newcommand*\indic[1]{\mathbbm{1}_{\{ #1 \}}}

\def\cL{{\mathcal L}}

\def\cP{{\mathcal P}}

\def\E{\mathbb{E}}

\def\N{\mathbb{N}}

\def\P{\mathbb{P}}
\def\R{\mathbb{R}}

\def\T{\mathbb{T}}

\def\d{\mathrm{d}}

\providecommand{\norm}[1]{\lVert#1\rVert}

\theoremstyle{plain}
\newtheorem{theorem}{Theorem}[section]
\newtheorem{lemma}[theorem]{Lemma}

\newtheorem{proposition}[theorem]{Proposition}

\newtheorem{remark}[theorem]{Remark}

\usepackage[symbol]{footmisc}

\usepackage{fancyhdr}
\date{\vspace{-1em}\normalsize{\today}}

\title{Renewal theorems in a periodic environment}

\author{Quentin Cormier\footnote{Inria, CMAP, École polytechnique, Institut Polytechnique de Paris. Email:
{\tt quentin.cormier@inria.fr}.}}

 \date{\today}
\begin{document}
\maketitle
\vspace{5pt}

\abstract{
	We study a renewal problem within a periodic environment, departing from the classical renewal theory by relaxing the assumption of independent and identically distributed inter-arrival times. Instead, the conditional distribution of the next arrival time, given the current one, is governed by a periodic kernel, denoted as $H$. The periodicity property of $H$ is expressed as $\mathbb{P}(T_{k+1} > t ~ |~ T_k) = H(t, T_k)$, where $H(t+T,s+T) = H(t, s)$.
	For a fixed time $t$, we define $N_t$ as the count of events occurring up to time $t$. The focus is on two temporal aspects: $Y_t$, the time elapsed since the last event, and $X_t$, the time until the next event occurs, given by $Y_t = t - T_{N_t}$ and $X_t = T_{N_{t}+1} - t$. The study explores the long-term behavior of the distributions of $X_t$ and $Y_t$.
}

\vspace{10pt}
\vspace{10pt}

\section{Introduction}
We are interested in a renewal problem in a periodic environment.
We consider $(T_k)_{k \geq 0}$, a sequence of increasing random variables modeling the successive arrival times of events.
Contrary to the classical renewal theory, we do not assume that the inter-arrival times 
$(T_{k+1}-T_k)_{k \geq 1}$ are 
independent and identically distributed random variables. Instead, 
we assume that the law of $T_{k+1}$ given $T_k$ is determined by a certain \textit{periodic} kernel $H(t, u)$:
\[  \forall k \geq 0, \quad \P(T_{k+1} > t \,|\, T_{k}) = H(t, T_k). \]
The kernel $H$ is \textit{periodic} in the sense that there exists a constant $T>0$ (the period) with
	\[ H(t, u) = H(t + T, u+T), \quad \forall t \geq u. \]
We assume that for all $u \in \R$, the function $t \mapsto H(t, u) \in C^1([u, \infty))$. We denote by  $K$ the associated density:
\begin{equation}
	\label{eq:def of K}
K(t,u) := - \frac{\d}{\d t} H(t, u). 
\end{equation}
One easily checks that $K$ is also $T$-periodic: $K(t+T,u+T) = K(t, u)$.
We write
\begin{equation}
	\label{eq:lambda t s}
	\lambda(t, u) := \frac{K(t, u)}{H(t, u)}. 
\end{equation}
By the previous assumption, $t \mapsto \lambda(t, u) \in C([u, \infty))$.
We assume that there exists two constants $0 < \lambda_{\min} < \lambda_{\max} < \infty$ such that
\[ \forall t \geq u, \quad  \lambda_{\min} \leq \lambda(t, u) \leq \lambda_{\max}. \]
By convention, we assume that $T_0 = s$ for some $s \in \R$ and that the first event occurs at time $T_1$. 
For $t \geq s$, we let $N^s_t$ be the number of events up to time $t$:
\[ N^s_t := \sup\{k \geq 0:  T_k \leq t \}. \]
We are interested by three quantities. The first one is the forward recurrent time, denoted $X^s_t$. This is the time we have to wait after $t$ to observe an event:
\[ X^s_t := T_{N^s_t + 1} - t. \]
The second one is the backward recurrent time $Y^s_t$, this is the time passed since the last event (or the age of the current component):
\[ Y^s_t := t - T_{N^s_t}. \]
Finally, we are interested by the instantaneous arrival time of events, that is:
\[ r(t, s) := \lim_{h \downarrow 0} \frac{1}{h} \P( N^s_{t+h} - N^s_t > 0). \] 
The goal of this work is to characterize the law of $X^s_t$ and $Y^s_t$, as well as the behavior of $r(t, s)$, for large $t$.

In the classical renewal theory,  the inter-arrival times are assumed to be independent and identically distributed. This corresponds here to a kernel $K$ satisfying $K(t, u) = \tilde{K}(t - u)$ for some density probability function $\tilde{K}$. Many different arguments have been developed in this setting, including direct probabilistic arguments, Volterra integral equations of convolution type or PDE techniques. We refer to \cite{zbMATH03168165, zbMATH01929532} for an in depth treatment of this classical problem. 
The periodic setting is studied in \cite{zbMATH03903308} by first extending the classical theory of renewal equations to ordered Banach spaces. The shape of the periodic asymptotic distributions of the backward process $(Y^s_t)$ is studied in \cite{zbMATH04075781}.
In \cite{Mischler05} and \cite{zbMATH07181472}, the authors studied the age structured PDE satisfied by the density of $Y^s_t$. They show existence of a family of $T$-periodic asymptotic profiles using the Krein–Rutman theorem, and prove the convergence at exponential rate using entropy inequalities \cite{Mischler05} or Harris's reccurrence theorem \cite{zbMATH07181472}. Theorem~\ref{th:backward} below summarize these results; we also provide the proof for completeness. 
To the best of our knowledge, Theorem~\ref{th:forward} is new. In particular, our main contribution is to identify the relevant asymptotic probability distributions and to provide a probabilistic interpretation of its shape. Finally, we mention an application of this work in neurosciences: consider a neuron of the type ``Integrate-and-Fire'' subject to a time-periodic input current. 
The statistics of the arrival times of the spikes of this neuron follow the structure described here, see \cite{CTV2}.

\section{Main results}
We first introduce some notations. 
Denote by $\T := \frac{\R}{T \mathbb{Z}}$ the circle of length $T$. For $t \in \R$, we write $t [T] \in \T$ for the value of $t$ modulo $T$.
We consider the following integral equations with unknown $\rho \in L^2(\T)$:
\begin{subequations}
	\label{eq:global-rho}
\begin{equation} 
	\label{eq:integral eq of rho:K}
	\rho(t) = \int_{-\infty}^t{ K(t, u) \rho(u) \d u },  \quad t \in \mathbb{R},
\end{equation}
\begin{equation}
	\label{eq:integral eq normalization rho:H}
	1 = \int_{-\infty}^t{ H(t, u) \rho(u) \d u}, \quad t \in \mathbb{R}.
\end{equation}
\end{subequations}
\begin{proposition}
	There exists a unique function $\rho \in L^1(\T)$ satisfying \eqref{eq:global-rho} for all $t \in \mathbb{R}$. Moreover, it holds that $\rho \in C(\T)$ and for all $t$,  $\lambda_{\min} \leq \rho(t) \leq \lambda_{\max}$.
\end{proposition}
The proof of this result is postponed in Section~\ref{sec:markov-chain-phases}.
The idea is to consider the Markov chain $(\Phi_k)_{k \geq 0}$ where $\Phi_k := T_k [T]$ is the \textit{phase} of $T_k$.
Then, one proves that the solution $\rho$ is proportional to the unique invariant probability measure of this Markov chain.

For $u \geq 0$ and $\phi, \varphi \in \R$, we set:
\begin{align}
\tilde{\nu}_\infty(u, \phi) &:= \frac{1}{T} \rho(\phi) H(u+\phi, \phi), 	\label{eq: def of nu infty phi and t} \\
\nu_\infty^\varphi(u) &:= T \tilde{\nu}_\infty(u, \varphi-u) = \rho(\varphi-u) H(\varphi, \varphi-u), 	\label{eq:def of nu infty varphi}
\end{align}
and
\begin{align}
	\tilde{\mu}_\infty(u, \phi) &:= \frac{1}{T} \int_{-\infty}^{\phi-u} { K(\phi, v) \rho(v) \d v}, \label{eq:def of mu infty phi t} \\
	\mu_\infty^\varphi(u) &:= T \tilde{\mu}_\infty(u, \varphi + u) = \int_{-\infty}^\varphi{ K(\varphi+u, v) \rho(v) \d v}. \label{eq:def of mu infty varphi}
\end{align}
\begin{lemma} 
	It holds that $\tilde{\nu}_\infty$ and $\tilde{\mu}_\infty$ are probability measures on $\R_+ \times \T$. Moreover, $\nu^\varphi_\infty$ and $\mu^\varphi_\infty$ are probability measures on $\R_+$ and $\nu^{\varphi+T}_\infty = \nu^\varphi_\infty$, $\mu^{\varphi+T}_\infty = \mu^\varphi_\infty$.
\end{lemma}
\begin{proof}
	Using Fubini and eq.~\eqref{eq:def of K}, 
\[		\int_0^\infty{ \mu_\infty^\varphi(t) \d t} = \int_0^\infty \int_{-\infty}^{\varphi} K(\varphi+t, u)\rho(u) \d u
							 = \int_{-\infty}^\varphi{ H(\varphi, u) \rho(u) \d u} 
							 = 1.
\]
In addition, using $\mu^\varphi_\infty(t) = T \tilde{\mu}_\infty(\varphi+t, t)$ and Fubini,
\[ 1 = \int_0^\infty{ \int_0^T \tilde{\mu}_\infty(\varphi+t, t) \d \varphi \d t} = \int_0^\infty{ \int_t^{T+t} \tilde{\mu}_\infty(\theta, t) \d \theta \d t}. \]
Note that $\phi \mapsto \tilde{\mu}_\infty(\phi, t)$ is $T$ periodic.
Therefore,  $\int_t^{T+t} \tilde{\mu}_\infty(\theta, t) \d \theta = \int_0^{T} \tilde{\mu}_\infty(\theta, t) \d \theta $. Altogether,	$\int_0^\infty \int_0^T \tilde{\mu}_\infty(\theta, t) \d \theta \d t = 1$. Similar arguments hold for $\nu^\varphi_\infty$ and $\tilde{\nu}_\infty$. 
\end{proof}
We denote by $d_{TV}(\nu, \mu)$ the standard total variation distance between two probability measures supported on $\R_+$.
We first recall the asymptotic results for $(Y^s_t)$ and $r(t, s)$, see \cite{Mischler05} and \cite{zbMATH07181472}.
\begin{theorem}
	\label{th:backward}
	There exists two constants $C, c > 0$ such that
	\begin{enumerate}
		\item For all $t \geq s$, $d_{TV}(\mathcal{L}(Y^s_t), \nu^t_\infty) \leq C e^{-c (t-s)}$.
		\item For all $t \geq s$, $|r(t, s) - \rho(t)| \leq C \lambda_{\max} e^{-c (t-s)}$.
		\item For all $a > 0$ and all $t \geq s$, $\left| \E (N^s_{t+a} - N^s_t) - \int_t^{t+a} \rho(u) \d u \right| \leq \frac{C \lambda_{\max}}{c} e^{-c(t-s)}$.
	\end{enumerate}
\end{theorem}
\begin{remark}
	We can choose $c = -\log(\beta)/T$ and  $C = 2 e^{cT}$, where $\beta := T \lambda_{\min} e^{-T \lambda_{\max}}$.
	The last inequality is a generalization of the key renewal theorem.
\end{remark}
The proof is given in Section~\ref{sec:backward} below for completeness. %
We then describe the long term behavior of the forward reccurence time $(X^s_t)_{t \geq s}$. Let $f(u) := e^{\frac{\lambda_{\min}}{2}u}$ and consider the following weighted total variation distance between two probability measures supported on $\R_+$:
\[  d^f_{\text{TV}}(\nu, \mu) := \int_{\R_+} { f(u) |\nu-\mu|(\d u)}. \]
Our main result is
\begin{theorem}
	\label{th:forward}
	There exists two constants $C, c > 0$ such that 
	\[ \forall t \geq s, \quad d^f_{TV}(\cL(X^s_t), \mu^t_\infty) \leq C e^{-c (t-s) }. \]
\end{theorem}
The proof is given in Section~\ref{sec:forward}. 

\section{Remarks on the asymptotic distributions}
\label{sec:pdmp}
We now explain using heuristic arguments the expressions of $\nu^\varphi_\infty$ and $\mu^\varphi_\infty$ which appear in the two theorems. We start with $\mu^\varphi_\infty$.
\subsubsection*{Forward Reccurence time}
There is a natural Piecewise Deterministic Markov process (PDMP) associated to the dynamics of the Forward Recurrence time.
Let $(\tilde{X_t}, \tilde{\Phi}_t)$ be the following PDMP:  between the jumps, $\tilde{\Phi}$ is constant and $\tilde{X}$ decays at rate 1 ($\frac{d}{dt} \tilde{X}_t = -1$). When $\tilde{X}$ reaches the value zero, a jump occurs for the two coordinates. The distribution of the jump is given by:
\[  \P(\tilde{X}_t \in \d u, \tilde{\Phi}_t \in \d \psi \, | \, \tilde{X}_{t-} = 0, \tilde{\Phi}_{t-} = \phi) = K(u+\phi, \phi) \delta_{\phi+u}(\d \psi) \d u. \]
The interpretation is the following: $\tilde{\Phi}_t$ is the \textit{phase} of the time next event; while $\tilde{X}_t$ is the time we have to wait before the next event occurs.
The successive jumping times of the PDMP define the sequence of arrival times $(\tilde{T}_k)_{k \geq 1}$.
\begin{proposition}
	The PDMP $(\tilde{X}_t, \tilde{\Phi}_t)_{t \geq 0}$ has a unique invariant probability measure on $\R_+ \times \T$ given by $\tilde{\mu}_\infty(u, \phi) \d u \d \phi$.
\end{proposition}
\begin{proof}
	We follow the approach of \cite{costaPDP} and consider the associated sampling chain  $(\tilde{\Phi}_{\tilde{T}_k}, \tilde{X}_{\tilde{T}_k})_{k \geq 0}$ of the process just after the jumps. 
	We shall see in Section~\ref{sec:markov-chain-phases} that $(\tilde{\Phi}_{\tilde{T}_k})$ has a unique invariant probability measure on $\T$,  $\pi(\phi) \d \phi$. 
	Therefore, $(\tilde{X}_{\tilde{T}_k}, \tilde{\Phi}_{\tilde{T}_k})_{k \geq 0}$ has a unique invariant probability measure given by:
\[  \mu_{MC}^\infty(\d u, \d \phi) =  K(\phi, \phi-u) \pi(\phi-u) \d \phi \d u. \]
Then, we use the result of \cite{costaPDP} to deduce that $\tilde{\mu}_\infty(u, \phi) \d u \d \phi$ is the unique invariant probability measure of $(\tilde{X}_t, \tilde{\Phi}_t)$.
\end{proof}
The connection between $X^s_t$ and this PDMP is the following.
Assume that $X^s_s$ is a random variable of law $K(u + s, s) \d u$.
Assume that at time $s$, $\tilde{X}_s = X^s_s$ and $\tilde{\Phi}_s = s+X^s_s$. Then the forward process is, for all $t \geq s$, $X^s_t = \tilde{X}_t$.
In addition, it holds that $a.s.,~\tilde{\Phi}_{t} = \tilde{X}_t + s$.
Therefore, $(X^s_{nT+\varphi})_{n}$ is a Markov chain. Indeed, it holds that 
\begin{align*}
	\P(X^s_{(n+1)T+\varphi} \in \cdot \,|\, X^s_{nT+\varphi} = u) &= \P(\tilde{X}_{(n+1)T+\varphi} \in \cdot \,|\, \tilde{X}_{nT+\varphi} = u, \, \tilde{\Phi}_{nT+\phi} = u + \varphi) \\
						  &= P_T(\cdot; (u, u + \varphi)),
\end{align*}
where $P_T(\cdot; (u, \phi))$ is the semi-group of the Markov process $(\tilde{X}_t, \tilde{\Phi}_t)_{t \geq 0}$.

We show formally that $\mu^\varphi_\infty(u) \d u := T \tilde{\mu}_\infty(u, u+\varphi)$ is the invariant probability measure of this Markov chain. 
Indeed, by the Birkhoff's ergodic theorem, we expect that
\begin{align*}
	\tilde{\mu}_\infty(u, \phi) \d u \d \phi & = \lim_{n \rightarrow \infty} \frac{1}{nT} \int_s^{s+nT} \indic{\tilde{X}_v \in \d u, \tilde{\Phi}_v \in \d \phi} \d v \\
					 &= \lim_{n \rightarrow \infty} \frac{1}{nT} \sum_{i = 0}^{n-1}\int_s^{s+T} \indic{\tilde{X}_{iT+v} \in \d u, \tilde{\Phi}_{iT+v} \in \d \phi} \d v \\
					 &= \lim_{n \rightarrow \infty} \frac{1}{nT} \sum_{i = 0}^{n-1} \indic{\tilde{X}_{iT+\phi-u} \in \d u} \d \phi = \frac{1}{T} \mu^{\phi-u}_\infty(u) \d u \d \phi.
\end{align*}
To obtain the third equality from the second one, we used that, provided that $\tilde{X}_{iT + s} = u$, we necessarily have $\tilde{\Phi}_{iT+s} = s+u$, or equivalently $s = \phi - u$. Hence, the integral simplifies.
Choosing $\varphi = \phi - u$, we find that $\mu^\varphi_\infty (u) \d u = T \tilde{\mu}_\infty(u, u+ \varphi) \d u$ as claimed.
Once again, this argument is only an heuristic: we give in Section~\ref{sec:forward} a complete proof that $\mu^\varphi_\infty$ is the unique invariant probability measure of $(X^s_{nT+\varphi})_{n \geq 0}$.
\subsubsection*{Backward recurrence time}
Similarly, consider the PDMP $(\tilde{Y}_t, \tilde{\Psi}_t)$ on $\R_+ \times \T$ defined as follows: between the jumps, $\tilde{\Psi}$ is constant and $\tilde{Y}$ grows at rate $1$ ($\frac{d}{dt} \tilde{Y}_t = 1$). Jumps occurs at rate $\lambda(\tilde{Y}_{t-} + \tilde{\Psi}_{t-}, \tilde{\Psi}_{t-})$. When a jump occurs, say at time $\tau$, it holds that $\tilde{Y}_\tau = 0$ and $\tilde{\Psi}_{\tau} = \tau$. The interpretation is the following: $\tilde{Y}_t$ is the time since the last event, while $\tilde{\Psi}_t$ is the phase corresponding to the time of the last event.
As previously, we can prove that $\tilde{\nu}_\infty$ is the unique invariant probability measure of $(\tilde{Y}_t, \tilde{\Psi}_t)$. 
In addition, $(Y^s_{nT+\varphi})_{n}$ is a Markov chain, and by a similar heuristic argument based on the Birkhoff's theorem, one expects $\nu^\varphi_\infty$ to be its invariant probability measure. This fact is proven rigorously in Section~\ref{sec:backward} below.
\section{The Markov chain of the phases}
\label{sec:markov-chain-phases}
For all $t, s \in [0, T]$, let $K^T(t, s) := \sum_{i \geq 0} K(t+i T, s)$.
We consider the phase of the $k$-th arrival time:
\[ \Phi_k := T_k [T]. \]
\begin{lemma}
	$(\Phi_k)_{k \geq 1}$ is a Markov chain on $\T$ with a transition kernel given by:
	\[ \P(\Phi_{k+1} \in \d t \,|\, \Phi_k) = \sum_{i \geq 0} K(t + iT, \Phi_k) \d t = K^T(t, \Phi_k) \d t. \]
\end{lemma}
\begin{proof}
	We write:
	\begin{equation}
	\label{eq:decomposition-phase}
	\forall k \geq 1, \quad T_k =: p_k T + \Phi_k, \quad \text{ with } \quad \Phi_k \in [0,T)  \quad \text{ and } \quad \Delta_{k+1} := p_{k+1} - p_k \in \mathbb{N}. 
	\end{equation}
	First, we note that by periodicity of $K$, 
	\[  \P(T_{k+1} - T_k \in \d t \,|\, T_k) = K(t+T_k,T_k) \d t =  K(t+\Phi_k, \Phi_k) \d t \]
	only depends on $\Phi_k$.
	Second, we have
	\begin{align*}
		\P(\Delta_{k+1} = i, \, \Phi_{k+1} \in \d \phi \,|\, \Phi_k) &= \P(T_{k+1} - T_k \in iT - \Phi_k + \d \phi \,|\, \Phi_k) \\
								       &= K(iT + \phi, \Phi_k) \d \phi.
	\end{align*}
Using that $K(u+iT,\Phi_k) = K(u, \Phi_k-iT)$ and summing on all the possible $i \in \N$ ends the proof.
\end{proof}
\begin{proposition}
	\label{prop:pi is the unique invariant measure of the phases}
	The Markov chain $(\Phi_k)_{k \geq 0}$ has a unique invariant probability measure with a continuous and periodic density $\pi \in C(\T)$. In addition, the solutions of \eqref{eq:integral eq of rho:K} in $L^1(\T)$ are given by the one dimensional vector space $\{ \beta \pi, \beta \in \R \}$. 
\end{proposition}
\begin{proof}
	Let $\delta := \inf_{t,s \in [0, T]} K^T(t, s) > 0$. We have
$\inf_{s \in [0, T]} K^T( \cdot , s)  \geq \delta T \text{Unif}(\cdot)$, 
 where \text{Unif} is the uniform measure on $[0, T]$. Therefore,  Doeblin's criterion is satisfied and $(\Phi_k)$ has a unique invariant probability measure $\pi(\d t)$. As
 $\pi(\d t) = \int_0^T K^T(t, s) \pi(\d s)$, it holds that $\pi$ has a continuous density $t \mapsto \pi(t) \in C(\T)$. We write $K^\T= L^1(\T) \rightarrow C(\T)$ for the operator 
 	\begin{equation}
		\label{eq:def-KT-operator}
y \in L^1(\T), \quad K^\T(y) = t \mapsto \int_0^T K^T(t, s) y(s) \d s. 
\end{equation}
Let $x \in L^1(\T)$ such that $x = K^\T(x)$. We prove that $x$ has a constant sign.
	We write $x_+$ for the positive part of $x$ and $x_-$ for its negative part. 
	Let $\theta = \min(\norm{x_+}_{L^1(\T)}, \norm{x_-}_{L^1(\T)})$. We have $K^\T(x_+)(t) \geq \delta \theta$ and $K^\T(x_-)(t) \geq \delta \theta$.
	So
	\begin{align*}
		\norm{x}_{L^1(\T)} &= \norm{K^\T(x_+) - K^\T(x_-)}_{L^1(\T)} \\
	& \leq \norm{K^\T(x_+) - \delta \theta}_{L^1(\T)} +  \norm{K^\T(x_-) - \delta \theta}_{L^1(\T)} = \norm{x}_{L^1(\T)} - 2 \delta \theta.  
\end{align*}
Therefore $\theta = 0$, and so $x$ has constant sign.
Finally, consider $\rho \in L^1(\T)$ a solution of \eqref{eq:integral eq of rho:K}. Using that $K(t, s) = 0$ for $s > t$, we have
\begin{align*}
	\rho(t) &= \int_{-\infty}^T K(t, s) \rho(s) \d s = \sum_{p \geq 0} \int_{-pT}^{T-pT} K(t, s) \rho(s) \d s \\
		&= \sum_{p \geq 0} \int_0^T K(t, u-pT) \rho(u) \d u = \int_0^T K^T(t, u) \rho(u) \d u.
\end{align*}
So $\rho = K^\T(\rho)$. Therefore, $\rho / (\int_{\T} \rho)$ is an invariant probability measure of $(\Phi_k)$, and by uniqueness, it holds that
$\rho(t) = (\int_{\T} \rho) \pi(t)$. \end{proof}
\begin{remark}
	So $\rho$, the solution of \eqref{eq:global-rho} exists and is unique: there exists a unique constant $\beta > 0$ such that
	$\rho = \frac{\pi}{\beta}$.
	Recall the decomposition \eqref{eq:decomposition-phase}. We verify that under the invariant measure $\mathcal{L}(\Phi_k) = \pi$, it holds that $\beta = \E (\Delta_{k+1})$.
\end{remark}

\section{Backward recurrence time}
\label{sec:backward}
\subsection{The canonical process and Renewal equations}
We introduce a slight generalization of the process $(Y^{s}_t)$. Let $\nu \in \cP(\R_+)$.
Let $N(\d u,\d z)$ be a Poisson measure on $\R_+ \times \R_+$ of intensity of Lebesgue measure $\d u \d z$. 
Let $(Y^{\nu, s}_t)_{t \geq 0}$ be the solution of 
\begin{equation} 
	\label{eq:canonical represention of Y}
	Y^{\nu, s}_t = Y^{\nu, s}_s + (t-s) - \int_s^t{ \int_{\R_+} Y^{\nu, s}_{u-} 1_{\{ z \leq \lambda(u, u - Y^{\nu, s}_{u-}) \}} N(\d u, \d z)}, 
\end{equation}
starting with law $Y^{\nu, s}_s \sim \nu$ at time $s$. This process $(Y^{\nu, s}_t)_{t \geq s}$ can be viewed as a ``canonical process'' associated to the renewal process: $Y^{\nu, s}_t$ is the time passed since the last jump. When $Y^{\nu, s}_{t-} \neq Y^{\nu, s}_t$, there is a jump of the process to zero, corresponding an event. 
Let $N^{\nu, s}_t$ be the number of jumps between $s$ and $t$:
\[ N^{\nu,s}_t = \int_s^t{ 1_{\{ z \leq \lambda(u, u-Y^{\nu,s}_{u-}) \}} N(\d u,\d z). } \]
We let $r^\nu(t, s) := \lim_{\Delta \downarrow 0} \frac{1}{\Delta} \P( \text{$Y^{\nu, s}$ has a jump between } [t, t+\Delta))$.
\begin{lemma}
	It holds that $r^\nu(t, s) = \E \lambda(t, t-Y^{\nu, s}_t)$.
\end{lemma}
\begin{proof}
It holds that
\begin{align*}
	\P(N^{\nu, s}_{t+\Delta}-N^{\nu, s}_t = 0) = \E \exp \left( - \int_t^{t+\Delta} \lambda(u, u-(Y^{\nu, s}_{t}+(u-t))) \right).
\end{align*}
Therefore, $\P(N^{\nu, s}_{t+\Delta}-N^{\nu, s}_t > 0) = \Delta \E \lambda(t, t-Y^{\nu, s}_{t})+ o(\Delta)$, proving the claim. 
\end{proof}
We also write 
\[ H^\nu(t, s) := \P( N^{\nu, s}_t = 0) = \int \exp\left(- \int_s^t \lambda(\theta, s-u) \d \theta \right) \nu(\d u) \]
and let $K^\nu(t, s)$ be the density of the first jump, that is: $K^\nu(t, s) := -\frac{d}{dt} H^\nu(t, s)$.
\begin{remark}
	This extends our previous notations in the following way:
	\[ H(t, s) = H^{\delta_0}(t, s), \quad K(t, s) = K^{\delta_0}(t, s), \quad r(t, s) = r^{\delta_0}(t, s), \quad Y^s_t = Y^{\delta_0, s}_t. \]
	where $\delta_0$ is the Dirac probability measure at zero.
\end{remark}
It is well known that the function $r^\nu(t, s)$ is the unique solution of a Volterra integral equation, also known as renewal equations. We recall this result without proof.
\begin{proposition}
	\label{prop:volterra int. eq. r}
	It holds that $r^\nu$ is the unique solution of the Volterra integral equations:
	\[ r^\nu(t, s) = K^\nu(t, s) + \int_s^t{ r(t, u) K^\nu(u, s) \d u} = K^\nu(t, s) + \int_s^t{ r^\nu(u, s) K(t,u)} \d u. \]
\end{proposition}
Finally, let $g: \R_+ \rightarrow \R_+$ be a non-negative measurable function. We have
\begin{lemma}
	\label{lem:renewa-structure-g}
	It holds that
	\[
		\E g(Y^{\nu, s}_t) = \int g(x + (t-s)) H^{\delta_x}(t, s) \nu (\d x) + \int_s^t K^\nu(u, s) \E g(Y^{\delta_0, u}_t) \d u.   
	\]
\end{lemma}
We now to study the long time behavior of this process $(Y^{\nu, s}_t)$.
Recall that $\nu^\varphi_\infty$ is defined by \eqref{eq:def of nu infty varphi}: 
	\[ \nu^\varphi_\infty(u)\d u := \rho(\varphi-u)H(\varphi, \varphi-u) \d u. \] 
\begin{proposition}
	\label{prop:invariant-periodic-law-of-Y}
	It holds that for all $t \geq s$, $\cL(Y^{\nu^s_\infty, s}_t) = \nu^t_\infty$.
\end{proposition}
\begin{proof}	\textit{Step 1.}
	Let $\nu := \nu^s_\infty$ and let $r^\nu(t, s) = \E \lambda(t, t-Y^{\nu, s}_t)$. We first verify that 
	\[ \forall t \geq s, \quad r^\nu(t, s) = \rho(t). \]
By Proposition \ref{prop:volterra int. eq. r}, $r^\nu$ is the unique solution of the Volterra integral equation
\[ r^\nu(t, s) = K^{\nu}(t,s) + \int_s^t{ K(t, u) r^\nu(u, s) \d u}. \] 
We verify that $\rho(t)$ solves the same equation. The result then follows by uniqueness.
We have:
\begin{align*}
	\rho(t) &= \int_{-\infty}^t{ K(t, u) \rho(u) \d u} = \int_{-\infty}^s{K(t, u) \rho(u) \d u} + \int_s^t{ K(t, u) \rho(u) \d u}.
\end{align*}
Therefore, it suffices to show that $\int_{-\infty}^s{ K(t, u) \rho(u) \d u} = K^\nu(t, s)$.
We have:
\begin{align*}
	K^\nu(t, s) &= \int_0^\infty{ \lambda(t, s-u) \exp \left(- \int_s^t{ \lambda(\theta, s-u) \d \theta}  \right) \nu(\d u)} \\
		    &= \int_0^\infty{ \lambda(t, s-u)  \exp \left(- \int_s^t{ \lambda(\theta, s-u) \d \theta}  \right) \rho(s-x) H(s, s-u) \d u } \\
		    &= \int_0^\infty { \lambda(t, s-u) H(t, s-u) \rho(s-u) \d u} = \int_{-\infty}^s{ K(t,u) \rho(u) \d u},
\end{align*}
as claimed. We used that $\exp \left(- \int_s^t{ \lambda(\theta, s-u) \d \theta}  \right)  H(s, s-u)  = H(t, s-u)$.\\
\textit{Step 2.} Let $g: \mathbb{R}_+ \rightarrow \R_+$ be a test function and $t > 0$. We verify that
$\E g(Y^{\nu,s}_t) =  \int_0^\infty{ g(u) \nu^t_\infty(u) \d u}$.
Let $\tau_t$ be the time of the last jump of $(Y^{\nu, s}_\cdot)$ before $t$, with the convention that $\tau_t = 0$ if there is no jump between $s$ and $t$. We have
\begin{align*}
	\E g(Y^{\nu,s}_t) &= \E g(Y^{\nu, s}_s + t - s) 1_{\{ \tau_t = 0 \}} + \E g(t-\tau_t)  1_{\{ \tau_t > 0 \}} =: A+B.
\end{align*}
It holds that
\begin{align*}
	A &= \int_0^\infty{ g(u + t -s ) H^{\delta_u}(t, s) \nu(\d u)} = \int_0^\infty{ g(u+t-s) H(t, s-u) \rho(s-u) \d u} \\
	  &= \int_{t-s}^\infty{ g(\theta) H(t, t-\theta) \rho(t-\theta) \d \theta}.
\end{align*}
In addition, for all $u \in (s, t)$, $\P(\tau_t \in \d u) = r^\nu(u, s) H(t, u) \d u$.  By Step 1, it holds that $r^\nu(u, s) = \rho(u)$. Altogether,
\[ B = \int_s^t{ g(t-u) \rho(u) H(t, u) \d u} = \int_0^{t-s}{ g(\theta) \rho(t-\theta) H(t, t-\theta) \d \theta}. \]
Therefore, $A+B = \int_0^\infty{ g(\theta) H(t, t-\theta) \rho(t-\theta) \d \theta} = \int_0^\infty{ g(\theta) \nu^t_\infty(\d \theta)}$. 
\end{proof}

\subsection{Proof of Theorem~\ref{th:backward}}
Let $s \in \R$ and $\varphi \in [0, T)$. For all $n \in \N$ such that $nT + \varphi \geq s$, we consider
\[ Y^\varphi_n :=  Y^{s}_{nT + \varphi}. \]
\begin{lemma}
	The sequence $(Y^\varphi_n)$ is a Markov chain on $\R_+$. In addition, given $g: \R_+ \rightarrow \R_+$ a non-negative measurable function, it holds that	\[ \E \left[ g(Y^\varphi_{n+1}) ~|~ Y^\varphi_n = y \right] = g(y+T) H^{\delta_y}(\varphi + T, \varphi) + \int_\varphi^{\varphi+T} K^{\delta_y}(v, \varphi) \E g(Y^{v}_{T+\varphi}) \d v.  \]
\end{lemma}
\begin{proof}
	We have $\E \left[ g(Y^\varphi_{n+1}) ~|~ Y^\varphi_n = y \right] = \E g(Y^{\delta_y, nT+\varphi}_{(n+1)T + \varphi})$. Using Lemma~\ref{lem:renewa-structure-g}, we find that
	\[ \E \left[g(Y^\varphi_{n+1}) ~|~ Y^\varphi_n = y \right] = g(y+T) H^{\delta_y}(\varphi+T, \varphi) + \int_{nT + \varphi}^{(n+1)T + \varphi} K^{\delta_y}(\theta, nT+\varphi) \E g(Y^{\theta}_{(n+1)T + \varphi}) \d \theta. \] 
	We make the change variable $v = \theta - nT$. Using that $\E g(Y^{v+nT}_{(n+1)T + \varphi}) = \E g(Y^v_{T+\varphi})$,  
	we deduce the stated formula, which does not depend on $n$. Therefore $(Y^\varphi_n)$ is a time-homogeneous Markov chain.
\end{proof}
\begin{lemma}
	$\nu^\varphi_\infty$ is an invariant probability measure of the Markov Chain $(Y^\varphi_{n})$.
\end{lemma}
\begin{proof}
	Assume that $\cL(Y^\varphi_n) = \nu^\varphi_\infty$. By Proposition~\ref{prop:invariant-periodic-law-of-Y},	$\cL(Y^\varphi_{n+1}) = \cL(Y^{\nu^\varphi_\infty, nT+\varphi}_{(n+1)T + \varphi}) = \nu^\varphi_\infty$. 
\end{proof}
\begin{lemma}
	Let $n_0$ such that $n_0T + \varphi \geq s$. 
	There exists a constant $\beta \in (0, 1)$ such that for all $n \geq n_0$, 
	\[ d_{TV}(\cL(Y^\varphi_{n}), \nu^\varphi_\infty) \leq \beta^{n-n_0} d_{TV}(\cL(Y^\varphi_{n_0}), \nu^\varphi_\infty). \]
\end{lemma}
\begin{proof}
	Denote by $y = Y^\varphi_n$. Let $g$ be a non-negative measurable function. By Lemma~\ref{lem:renewa-structure-g},	\[ \E g(Y^\varphi_{n+1}) = \E g(Y^{\delta_y, \varphi}_{\varphi+T}) \geq \int_{\varphi}^{\varphi + T} K^{\delta_y}(v, \varphi) \E g(Y^{\delta_0, v}_{T+\varphi}) \d v.  \]
	By assumption,
	$ T K^{\delta_y}(u, \varphi) \geq T \lambda_{\min} e^{-\lambda_{\max}T } =: \beta > 0$. 
	Following the argument of \cite[Sec. 3.4]{zbMATH07181472}, define the probability measure $\nu \in \cP([0, T])$ by
	\[ \nu(A) := \frac{1}{T} \int^{\varphi+T}_\varphi \P( Y^{\delta_0, u}_{T+\varphi} \in A) \d u. \]
	It holds that $\E g(Y^\varphi_1) \geq \beta \nu(g)$.
	Therefore, Doeblin's theorem applies, giving the result.
\end{proof}
We now give the proof of Theorem~\ref{th:backward}. Let $t \geq s$, we write $t = nT + \varphi$. Let $n_0$ be the smallest integer such that $n_0T + \varphi \geq s$, so that $t-s \leq (n-n_0) T + T$.
We choose $c = - \log(\beta) / T$ and $C = 2 e^{cT}$. We have
\[ d_{TV}(\cL(Y^s_t), \nu^t_\infty) = d_{TV}(\cL(Y^\varphi_n), \nu^\varphi_\infty) \leq 2 \beta^{n-n_0} \leq C e^{-c(t-s)}. \]
Finally, recall the representation $r(t, s) = \E \lambda(t, t-Y^{s}_t) = \E \lambda(\varphi, \varphi - Y^\varphi_n)$.
The function $\psi(y) := \lambda(\varphi, \varphi-y)$ is bounded by $\lambda_{\max}$, so
\[ \left| \E \psi(Y^\varphi_n) - \int_0^\infty \psi(y) \nu^\varphi_\infty(\d y)  \right| \leq \lambda_{\max} \beta^{n-n_0}. \]
We verify that $\int_0^\infty \psi(y) \nu^\varphi_\infty(\d y) = \rho(\varphi) = \rho(t)$, therefore, $|r(t, s) - \rho(s)| \leq C \lambda_{\max} e^{-c (t-s)}$.
The last stated inequality follows from the representation $\E [N^s_{t+a} - N^s_t] = \int_t^{t+a} r(u, s) \d u$.
\qed
\section{Forward recurrence time}
\label{sec:forward}
\subsection{Invariant measure of the sampled chain}
Recall that $X^s_t := T_{N^s_t + 1} - t$.
Let $\varphi \in [0, T)$ and $n \geq 0$ such that $nT+\varphi \geq s$. We set 
\[ X^\varphi_n := X^s_{nT+\varphi}. \]
We have seen that $(X^\varphi_n)_{n \geq 0}$ is a Markov Chain on $\mathbb{R}_+$ with transition probability given by
\begin{align*}
	\P(X^\varphi_{n+1} \in \cdot \,|\, X^\varphi_n = u) &= \P(\tilde{X}_{(n+1)T + \varphi} \in \cdot \,|\, \tilde{X}_{nT+\varphi} = u, \, \tilde{\Phi}_{nT+\varphi} = \varphi + u) \\
						  &= P_T(\cdot; (u, \varphi + u)),
\end{align*}
where $P_T(\cdot; (u, \varphi))$ is the semi-group of the PDMP $(\tilde{X}_t, \tilde{\Phi}_t)_{t \geq 0}$ introduced in Section~\ref{sec:pdmp}.
\begin{proposition}
	The probability measure $\mu^\varphi_\infty(t)dt$, defined by \eqref{eq:def of mu infty varphi} is an invariant probability measure of the Markov Chain $(X^\varphi_n)$.
\end{proposition}
The result follows from the following steps.
\subsubsection*{Transition probability of the Markov Chain}
The first step is to explicit the transition probability kernel of the Markov chain $(X^\varphi_n)$.
Let $g: \R_+ \rightarrow \R$ be a continuous and bounded test function. We write for all $t \geq 0$:
\[ \cP g(t) := \E [g(X^{\varphi}_{n+1}) \,|\, X^\varphi_n = t]. \]
Provided that $X^\varphi_n = t > T$, no jump occurs and $X^\varphi_{n+1} = t-T$ and so $\cP g(t) = g(t-T)$.
We therefore consider the case $t \in [0, T]$. Let:
\[ g^K(t) := \int_0^\infty {g(u) K(u+T+\varphi, \varphi + t) \d u}.\]
\begin{lemma}
	For $t \in [0, T]$, $\cP g$ solves the integral equation
	\begin{equation}
		\label{eq:integral equation pt}
		\cP g(t) = g^K(t) + \int_t^T{ K(\varphi + u, \varphi+t) \, \cP g (u) \d u}, \quad t \in [0, T]. 
	\end{equation}
\end{lemma}
\begin{proof}
	During this proof, we use the notation:
	\[\E_{(s, u)} g(\tilde{X}_{T+\varphi}) := \E[g(\tilde{X}_{T+\varphi}) \,|\, \tilde{X}_s = u, \tilde{\Phi}_s = s+u]. \]
	With this notation, it holds that $\cP g(t) = \E_{(\varphi, t)} g(\tilde{X}_{T+\varphi})$.
	Because $t \leq T$, $\tilde{X}$ starting with value $t$ at time $\varphi$ has a jump at time $t + \varphi$. Let $\Delta$ be the size of this jump.
	We have by the Markov property at time $\varphi+t$:
	\begin{align*}
		\cP g(t) &= \E_{(\varphi, t)} g(\tilde{X}_{T+\varphi}) =  \E \E_{(\varphi+t, \Delta)} g(\tilde{X}_{T+\varphi}) = \E \E_{(\varphi, \Delta+t)} g(\tilde{X}_{T+\varphi}) \\
		    &= \E \indic{\Delta > T-t} \E_{(\varphi, \Delta+t)} g(\tilde{X}_{T+\varphi}) + \E \indic{\Delta \leq T-t} \E_{(\varphi, \Delta+t)} g(\tilde{X}_{T+\varphi}) =: A + B.
	\end{align*}
Note that $A = \E (g(\Delta - (T-t)) 1_{\{\Delta > T-t\}} )$.
	As the law of $\Delta$ is 
	$\mathcal{L}(\Delta)(\d s) = K(s + t+\varphi, t+\varphi)\d s$,  \[ A = \int_{T-t}^\infty{ g(s - (T-t)) K(s + t + \varphi, t+\varphi) \d s} = g^K_t. \]
In addition, 
\[ B = \E \indic{\Delta \leq T-t} \cP g (\Delta+t)
	= \int_0^{T-t} K(s+t+\varphi, t+\varphi) \cP g (s+t) \d s.
\]
The change of variable $u = s+t$ gives the stated result.
\end{proof}
\subsubsection*{Resolution of the integral equation}
Recall that for $t \geq s$,  $r(t, s)$ is the solution of the Volterra integral equation
\begin{equation}
	\label{eq:volterra eq r and K}
	r(t, s) = K(t, s) + \int_s^t{ r(t, u) K(u, s)  \d u}.
\end{equation}
\begin{lemma}
	\label{lem:resolvent form of pt}
	It holds that for all $t \in [0, T]$,
	\[ \cP g(t) = g^K(t) + \int_t^T{ r(\varphi+u, \varphi+t) g^K(u) \d u}. \]
\end{lemma}
\begin{proof}
	We verify that $\tilde{p}(t) := g^K(t) + \int_t^T{ r(\varphi+u, \varphi+t) g^K(u) \d u}$ solves \eqref{eq:integral equation pt}. As by a classical argument, it holds that \eqref{eq:integral equation pt} has a unique solution, this proves the result. We have
\begin{align*}
	\int_t^T{K(\varphi+u, \varphi+t) \tilde{p}(u) \d u} &= \int_t^T{K(\varphi+u, \varphi+t) g^K(u) \d u} \\
						     & \quad + \int_t^T \int_u^T K(\varphi+u, \varphi+t) r(\varphi+\theta, \varphi+u) g^K(\theta) \d \theta \d u \\
						     &=: A+B.
\end{align*}
By Fubini and eq.~\eqref{eq:volterra eq r and K}, we have: 
\begin{align*}
B &= \int_t^T {g^K(\theta) \int_t^\theta  r(\varphi+\theta, \varphi+u) K(\varphi+u, \varphi+t) \d u\d \theta} \\
	&= \int_t^T { g^K(\theta) \left[ r(\varphi+\theta, \varphi+t) - K(\varphi+\theta, \varphi+t) \right] \d \theta}.
\end{align*}
Therefore, $A+B = \int_t^T r(\varphi+u, \varphi+t) g^K(u) \d u = \tilde{p}(t) - g^K(t)$ and so $\tilde{p}$ solves \eqref{eq:integral equation pt}.
\end{proof}
\subsubsection*{Final computations}
Let $\nu(u) \d u$ be a probability measure on $\R_+$ with density $\nu(u)$. We have
\begin{align*}
	& \int \cP g (u) \nu (\d u)  = \int_T^\infty{ g(u-T) \nu(u)\d u} + \int_0^T { \cP g(t) \nu(t) \d t} \\
				& \qquad  = \int_0^\infty { g(t) \nu(t+T) \d t} + \int_0^T { g^K(t) \nu(t) \d t} + \int_0^T{ \int_t^T {r(\varphi+u, \varphi+t)g^K(u) \d u }  \nu(t) \d t}.
\end{align*}
To simplify the computations, let $K^{[u], \varphi}(t,s) := K(u + \varphi+t, \varphi+s)$ and $r^\varphi(t, s) := r(t+\varphi, s+\varphi)$. We also use the notation
\[ (A*B)(t, s) = \int_s^t A(t, u) B(u, s) \d u \quad \text{ and } \quad (A*\nu)(t) = \int_0^t A(t, s) \nu( \d s). \]
Using Fubini, we have
\[  \int \cP g (u) \nu(u) \d u = \int_0^\infty{ g(u) \left[ \nu(u+T) + (K^{[u], \varphi}*\nu) (T) + (K^{[u], \varphi} * r^\varphi * \nu)(T) \right] \d u}. \]
This shows that $\nu$ is an invariant probability measure of $(X^\varphi_n)$ provided that 
\begin{equation}
	\label{eq:characterization invariant probability measures}
\nu(u+T) + (K^{[u], \varphi}*\nu) (T) + (K^{[u], \varphi} * r^\varphi * \nu)(T)  = \nu(u), \quad \forall u \geq 0. 
\end{equation}
Recall that $\rho(t)$ is the solution of \eqref{eq:global-rho}. We also define $\rho^\varphi(t) := \rho(\varphi+t)$ and $K^\varphi(t, s) := K(\varphi+t, \varphi+s)$. Recall that $\mu^\varphi_\infty$ is defined by \eqref{eq:def of mu infty varphi}:
\[ \mu^\varphi_\infty(u) = \int_{-\infty}^\varphi{ K(\varphi + u, \theta) \rho(\theta) \d \theta.} \]
\begin{lemma}
	The probability measure $\mu^\varphi_\infty(u) du$ satisfies:
	\begin{enumerate}
		\item $\mu^\varphi_\infty(u+T) = \mu^\varphi_\infty(u) - (K^{[u],\varphi} * \rho^\varphi)(T)$.
		\item $\mu^\varphi_\infty = \rho^\varphi - K^\varphi * \rho^\varphi$.
		\item $K^\varphi * \rho^\varphi = r^\varphi * \mu^\varphi_\infty$.
	\end{enumerate}
\end{lemma}
\begin{proof}
	First, we have:
	\begin{align*}
		\mu^\varphi_\infty(u+T) &= \int_{-\infty}^\varphi { K(\varphi+u+T, \theta) \rho(\theta) \d \theta} = \int_{-\infty}^\varphi { K(\varphi+u, \theta-T) \rho(\theta-T) \d \theta} \\
				     &= \int_{-\infty}^{\varphi-T} { K(\varphi+u, s) \rho(s) ds} = \mu^\varphi_\infty(u) - \int^{\varphi}_{\varphi-T} {K(\varphi+u, s) \rho(s) \d s}.
	\end{align*}
	Therefore, by the change of variables $\theta = s+ \varphi - T$, we obtain the first equality.
	Second, 
	\begin{align*}
		\mu^\varphi_\infty(u) &= \int_{-\infty}^\varphi { K(\varphi+u, \theta) \rho(\theta) \d \theta} \\
				  &= \int_{-\infty}^{\varphi+u}{K(\varphi+u, \theta) \rho(\theta) \d \theta} - \int^{\varphi+u}_{\varphi} {K(\varphi+u, \theta) \rho(\theta) \d \theta}\\
				  &= \rho(\varphi+u) - \int_0^u{ K(\varphi+u, s+\varphi)\rho(s+\varphi) \d s} \\
				  &= \rho^\varphi(u) - (K^\varphi * \rho^\varphi)(u).
	\end{align*}
Third, using the second point, we have $r^\varphi * \mu^\varphi_\infty = r^\varphi * \rho^\varphi - r^\varphi * K^\varphi * \rho^\varphi$. 
In addition, we have $r^\varphi = K^\varphi + r^\varphi * K^\varphi$, therefore $r^\varphi * \mu^\varphi_\infty = K^\varphi * \rho^\varphi$.
\end{proof}
	Altogether, $\mu^\varphi_\infty$ satisfies \eqref{eq:characterization invariant probability measures} and so $\mu^\varphi_\infty$ is an invariant probability measure of $(X^\varphi_n)$.
\subsection{Application of the Harris's ergodic theorem}
We now verify that $(X^\varphi_n)$ satisfies the assumptions of the Harris's ergodic theorem \cite{zbMATH06071108}.
This implies uniqueness of the invariant probability measure and convergence at an exponential rate in total variation norm.
Recall that $f(t) := e^{ \lambda_{\min} t /2}$. We also set
\[ \gamma := e^{-\lambda_{\min} T / 2} \quad \text{ and } \quad \kappa := 2 \left( \frac{\lambda_{\max}}{\lambda_{\min}} + \frac{\lambda^2_{\max}}{\lambda^2_{\min}}\right).  \]
\begin{lemma}
	\label{lem:lyapunov}
	The function $f$ is a Lyapunov function for $(X^\varphi_n)$ in the sense that \[ \forall t \geq 0, \quad \cP f(t) \leq \gamma f(t) + \kappa. \]
\end{lemma}
\begin{proof}
	When $t > T$, $\cP f(t) = f(t-T) = \gamma f(t)$. When $t \in [0, T]$, it holds that
	\[ f^K_t := \int_0^\infty{ f(u) K(u+T+\varphi, \varphi+t) \d u} \]
	satisfies
	$f^K_t \leq \frac{2 \lambda_{\max}}{\lambda_{\min}} e^{-\lambda_{\min}(T-t)}$. 
We have $r(t, s) \leq \lambda_{\max}$. Therefore, by Lemma~\ref{lem:resolvent form of pt},
	\[ \sup_{t \in [0, T]} \cP f(t) \leq f^K_t + \int_t^T{ \lambda_{\max} f^K_u du} \leq \kappa.  \]
\end{proof}
Recall that the weighted total variation between two probability measures on $\R_+$ is $d^f_{TV}(\nu, \mu) = \int_{\R_+} { f(t) |\nu-\mu|(\d t)}$.
We denote by $\cP_*$ the adjoint operator of $\cP$, such that $\int \cP g(t) \nu (\d t) = \int g(t) \cP_* \nu (\d t)$, or equivalently:
\[  \mathcal{P}_* \nu = \mathcal{L}(X^\varphi_{n+1}), \quad \text{ provided that } \quad  \mathcal{L}(X^\varphi_n) = \nu. \]
\begin{lemma}
	For all $\nu, \mu \in \cP(\R_+)$, it holds that:
\[ d^f_{\text{TV}}(\mathcal{P}_*(\nu), \mathcal{P}_*(\mu))  \leq (\kappa + \gamma) d^f_{\text{TV}}(\nu, \mu). \]
\end{lemma}
\begin{proof}
	This follows from Lemma~\ref{lem:lyapunov} above.
\end{proof}
\begin{lemma}
	Assume that $T$ is large enough such that $\gamma = e^{-\lambda_{\min}T / 2} \leq \frac{1}{2}$ and $T > \frac{2 \log(8 \kappa)}{\lambda_{\min}}$. 
	Then for $\nu(\d t) = \lambda_{\max} e^{-\lambda_{\max}t} \d t$, it holds that for some $\alpha > 0$,
	\[
		\inf_{\{ u \geq 0 : f(u) \leq \frac{4 \kappa}{1-\gamma} \} } \P(X^\varphi_{n+1} \in A~ |~ X^\varphi_n = u) \geq \alpha \nu(A).
	\]
\end{lemma}
\begin{proof}
	Let $u \geq 0$ such that $f(u) \leq \frac{4 \kappa}{1-\gamma} \leq 8 \kappa$. So, $e^{\lambda_{\min} u  /2 } \leq 8\kappa$ 
	and so $u \leq \frac{2 \log(8 \kappa)}{\lambda_{\min}} < T$. Therefore, the process $(\tilde{X}, \tilde{\Phi})$, starting at time $\varphi$ with values $(u, u+\varphi)$ has a jump at time $\varphi+u<\varphi+T$. Let $\Delta$ be the size of the jump. It holds that
	\begin{align*}  \P(\tilde{X}_{T+\varphi} \in \d t ~| ~\tilde{X}_\varphi = u, \tilde{\Phi}_{\varphi} = u + \varphi) & \geq \P(\Delta \in \d t + T-u) \\
		& = K(u+\varphi+ (t+T-u), u+\varphi) \d t \\
		&\geq \lambda_{\min} e^{-\lambda_{\max} (t+T-u)}\d t. 
	\end{align*}
	This gives the stated result with $\alpha = \frac{\lambda_{\min}}{\lambda_{\max}} e^{-\lambda_{\max}T}$.
\end{proof}
Finally, we give the proof of Theorem~\ref{th:backward}.
Note that without loss of generality, we can assume that $T > \tfrac{2 \log(8 \kappa)}{\lambda_{\min}}$
and that $\gamma = e^{-\lambda_{\min}T / 2} \leq \frac{1}{2}$ (otherwise, replace $T$ by $pT$ for some $p \in \mathbb{N}$ large enough). Let $t \geq s$ be fixed. Write $t = nT+\varphi$, $\varphi \in [0, T)$. Let $n_0$ be the smallest integer such that $n_0 T + \varphi > s$.
In view of the previous lemma, the Harris's ergodic theorem \cite[Th. 1.2]{zbMATH06071108} applies, and so there exists $C > 0$ and $\theta \in (0, 1)$ such that
\[ \forall n \geq n_0, \quad  d^f_{TV}(\cL(X^s_{nT+\varphi}), \nu^\varphi_\infty) \leq C \theta^{n-n0} d^f_{TV}(\cL(X^s_{n_0T+\varphi}), \mu^\varphi_\infty).  \]
We have $\cL(X^s_s) = K(s+\cdot, s)$. As $ \sup_{s,\varphi \in \R} \sup_{u \geq 0} [K(s+u, s) + \nu^\varphi_\infty(u)]  e^{\lambda_{\min} u} < \infty$, there exists a constant $C$ independent of $s$ and $\psi$ such that 
\[ d^f_{TV}(\cL(X^s_s), \mu^\varphi_\infty) \leq C. \]
In addition, as $n_0T + \varphi - s < T$, there exists another constant $C$, independent of $s$ such that
\[ d^f_{TV}(\cL(X^s_{n_0T+\varphi}),  \cL(X^s_s)) \leq C. \]
Therefore, we have for another constant $C$:
\[ \forall n \geq n_0, \quad \quad  d^f_{TV}(\cL(X^s_{nT+\varphi}), \nu^\varphi_\infty) \leq C \theta^{n-n0}. \] 
The conclusion of Theorem~\ref{th:backward} easily follows. \qed

\bibliographystyle{abbrvnat}
\bibliography{biblio}
\end{document}